\documentclass[12pt]{amsart}

\usepackage{amsmath}
\usepackage{amsfonts}
\usepackage{amsthm}
\usepackage{amssymb}
\usepackage[english]{babel}
\usepackage[ansinew]{inputenc}
\usepackage[all]{xy}

\theoremstyle{definition}

\newtheorem{defi}{Definition}[section]
\newtheorem{remark}[defi]{Remark}
\newtheorem{notation}[defi]{Notation}

\theoremstyle{plain}

\newtheorem{theorem}[defi]{Theorem}
\newtheorem{corollary}[defi]{Corollary}
\newtheorem{lemma}[defi]{Lemma}
\newtheorem{proposition}[defi]{Proposition}

\input xy
\xyoption{all}

\begin{document}

\title[Fundamental group scheme of rationally chain connected varieties]{On the
fundamental group scheme of rationally chain connected varieties}

\author[M. Antei]{Marco Antei}

\address{Laboratoire J.A.Dieudonn\'e, UMR CNRS-UNS No 7351
Universit\'e de Nice Sophia-Antipolis, Parc Valrose,
06108 NICE Cedex 2, France}

\email{Marco.ANTEI@unice.fr}

\author[I. Biswas]{Indranil Biswas}

\address{School of Mathematics, Tata Institute of Fundamental
Research, Homi Bhabha Road, Bombay 400005, India}

\email{indranil@math.tifr.res.in}


\begin{abstract}
Let $k$ be an algebraically closed field.
Chambert-Loir proved that the \'etale fundamental group of a proper normal
rationally chain connected variety over $k$ is finite. 
We prove that the fundamental group scheme of a proper normal rationally chain connected
variety over $k$ is finite too. In particular, the
fundamental group scheme of a Fano variety is finite.
\end{abstract}

\maketitle

\section{Introduction}

Let $X$ be a scheme over an algebraically closed field $k$ of
characteristic $p$, with $p\, >\, 0$. Fix a $k$-point $x$ of $X$. Nori
introduced the notion of fundamental group scheme $\pi(X,x)$ in \cite{N1}
and further developed it in \cite{N2}. Since then the
fundamental group scheme is being
studied and in the process has turned into
an important tool in algebraic geometry of positive
characteristic. In \cite{N2} Nori proves that
$\pi(X,x)$ is trivial for proper rational normal varieties. More generally,
$\pi(X,x)\,=\, 0$ if $X$ is separably rationally connected \cite{Bi}.
Zhu proves that a general Fano (proper, smooth, connected with
ample anticanonical bundle) hypersurfaces in projective spaces are
separably rationally connected \cite{Zh}. Therefore, the fundamental
group schemes of general Fano hypersurfaces in a projective space
are trivial.

In \cite{CL} Chambert-Loir proves that every proper rationally
chain connected normal variety has finite \'etale
fundamental group, and its order is coprime to $p$
(the characteristic of $k$) \cite{CL2}. This result can also be obtained as a
consequence of \cite[Theorem 1.5]{Ko} and \cite[Theorem 1.6]{Ko}. Shioda gave an example of a rationally
connected variety over a field of characteristic $p\,\neq\, 5$ whose \'etale
fundamental group is $\mathbb{Z}/5\mathbb{Z}$ \cite{Sh}. Examples of rationally chain connected varieties whose local 
fundamental group scheme is not trivial are also known.
For example, a supersingular Enriques surface $E$ over an algebraically 
closed field of characteristic $2$ is unirational (see
\cite[Corollary 1.3.1]{CD}), hence it is rationally 
chain connected. It is known that there exists a nontrivial
$\alpha_2$--torsor over $E$ (see \cite[Chapter I, \S 3]{CD}). 

We prove the following (see Theorem \ref{teoMAIN} and Remark \ref{remFANO}):

\begin{theorem}
Let $k$ be an algebraically closed field and $X$
a proper normal, rationally chain connected $k$--scheme. Let $x\,\in\, X(k)$ be a point.
Then the fundamental group scheme $\pi(X,x)$ is finite.
\end{theorem}

The strategy of the proof is similar to that in \cite{CL}, adapted to the 
new setting.

\section{Preliminaries}\label{sez:PREM}

We will write $\pi(X)$ instead of $\pi(X,x)$ to simplify the notation. 
However all the schemes for which we will compute the fundamental group 
scheme are meant to be pointed and all the morphisms between them take the 
marked point in the domain space to the marked point in the target space. 
The same convention will be applied to torsors: we assume 
they are pointed and morphisms between them take the marked point in the 
domain space to the marked point in the target space.

Let $k$ be an algebraically closed field of any characteristic.
A proper variety $X$ over $k$ is said to be rationally chain connected if for every algebraically
closed field $\Omega$ containing $k$, for any two points in $X(\Omega)$ there is a
proper and connected curve passing through them such that its normalization is a
disjoint union of projective lines. If this union consists of only one projective line
we say that $X$ is rationally connected.

Let $X$ be a rationally chain connected variety over $k$.
We recall \cite[Lemma 1]{CL} and we sketch its proof:

\begin{lemma}\label{lemCL1}
Let $k\,\subseteq\, \Omega$ be a field extension
where $\Omega$ is algebraically closed. Let $$F_{\Omega}\,:\, \mathbb{P}_{\Omega}^1
\,\longrightarrow\, X_{\Omega}$$ be a rational curve of $X_{\Omega}$. Let
$x_0\,:=\,F_{\Omega}(0)$ and $x_{\infty}\,:=\,F_{\Omega}(\infty)$ be points of
$X_{\Omega}$ then let $V_0$ and $V_{\infty}$ be their Zariski closure in $X$.
Then there exist a normal integral $k$-scheme $T$, a morphism
$$F\,:\, \mathbb{P}_{T}^1\,\longrightarrow\, X$$ such that the morphisms defined as
$$F_0(t)\,:=\,F(0,t)\,:\,T\,\longrightarrow\, X\ ~ \text{ and }~\ F_{\infty}(t)\,:=\,
F(\infty,t)\,:\,T\,\longrightarrow\, X$$
are dominant over $V_0$ and $V_{\infty}$ respectively.
\end{lemma}

\begin{proof}
There exists a finitely generated $k$-algebra $k\,\subseteq\, A$ contained
in $\Omega$, and there is a morphism $$F_{A}\,:\,
\mathbb{P}_{A}^1\,\longrightarrow\, X_{A}$$ such that $F_{\Omega}\,=\,F_A
\otimes_A
\Omega$. We set $T\,:=\,{\rm Spec}(A)$ that we assume to be normal (otherwise we replace it with a finite extension). We now consider the morphism
$F\,:\, \mathbb{P}_{T}^1\,\longrightarrow\, X$ obtained from $F_{A}$ after composing with the projection $X_A\,\longrightarrow\, X$. In \cite[Lemma 1]{CL} it has been proved that $F_0$ and $F_{\infty}$ are
dominant over $V_0$ and $V_{\infty}$; we briefly recall this last part for the
convenience of the reader: We study $F_{0}$ (it will
be the same for $F_{\infty}$). The
image by $F_{0}$ of the generic point of ${\rm Spec}(A)$ is the generic
point of $V_0$. Since $V_0$ is closed in $X$, the inverse image
$F_{0}^{-1}(V_0)$ is closed in ${\rm Spec}(A)$ and dense. Thus
$F_{0}^{-1}(V_0)$ coincides with
${\rm Spec}(A)$, and hence the image of $F_{0}$ is contained in $V_0$ and contains
its generic point. Therefore it contains an open dense subset of $V_0$. 
\end{proof}

\section{The main theorem}

The following lemma is well-known. We include
a short proof of it for the convenience of the reader.

\begin{lemma}\label{lemGROUPS}
Let $G$ be a finite $k$-group scheme and let $G^{\text{\'et}}$ and $G^{\rm loc}$ be
respectively the maximal \'etale quotient and the maximal connected quotient. Then the
natural morphism $$\alpha\,:\,G\,\longrightarrow\, G^{\text{\'et}}\times G^{\rm loc}$$
is faithfully flat. 
\end{lemma}

\begin{proof}
The field being perfect the reduced subscheme $G_{\text{red}}$ is a subgroup
scheme of $$N^{\text{loc}}\,:=\,\text{kernel} (G\to G^{\text{loc}})\, ,$$ while the
connected component $G^0$ of $G$ is $\text{kernel}(G\to 
G^{\text{\'et}})$. If $\alpha$ is not faithfully flat we can factor it as in
the following diagram:
$$\xymatrix{ & G\ar@{->>}[ldd]\ar@{->>}[d]^{q}\ar@{->>}[rdd] & \\ &
G'\ar@{->>}[ld]\ar@{^{(}->}[d]^{j}\ar@{->>}[rd] & \\G^{\text{\'et}} &
G^{\text{\'et}}\times G^{\text{loc}}\ar@{->>}[l]\ar@{->>}[r] & G^{\text{loc}} }
$$
where $q\,:\,G\,\twoheadrightarrow\, G'$ is faithfully flat and $j\,:\,G'\,
\hookrightarrow\, G^{\text{\'et}}\times G^{\text{loc}}$ is a closed immersion. Clearly
$G^{\text{\'et}}$ and $G^{\text{loc}}$ are still the maximal \'etale
quotient and the maximal connected quotient of $G'$ respectively. So we can
assume $\alpha$ is a closed immersion. Therefore, the lemma is equivalent to the
assertion that $\alpha$ is an isomorphism.

{}From \cite[\S~6, Ex. 9]{WW} it follows that $G^{\text{\'et}}$
is isomorphic to a subgroup-scheme of $G$ which we identify with
$G^{\text{\'et}}$, so in particular $G^{\text{\'et}}\,\leq\, 
G_{\text{red}}\,\leq\, N^{\text{loc}}$. Therefore, we have 
$$\vert G^{\text{\'et}}\vert\vert G^{\text{loc}}\vert\,=\,
\vert G^{\text{\'et}}\vert\frac{\vert G\vert }{\vert N^{\text{loc}}\vert
}\,\leq\,\vert G^{\text{\'et}}\vert\frac{\vert G\vert }{\vert G^{\text{\'et}}
\vert } \,=\, \vert G\vert $$
which implies that $\alpha$ is an isomorphism.
\end{proof}

We recall that when $X$ is a reduced and connected scheme over a field $k$
then the fundamental group scheme can be defined. In this case a finite $G$-torsor $Y\,\longrightarrow\, X$ is called Nori-reduced if the canonical morphism $\pi(X)
\,\longrightarrow\, G$ is faithfully flat. 

\begin{lemma}\label{lemTORS}
Let $X$ be a connected and reduced scheme over $k$. Let $G$ (respectively,
$H$) be a finite local (respectively, finite \'etale) $k$--group
scheme. Let $$Y\,\longrightarrow\, X$$ and $T\,\longrightarrow\, X$ be a $G$--torsor and
an $H$--torsor respectively. We assume that both $Y$ and $T$ are Nori-reduced. Then
the $H\times G$--torsor $$T\times_X Y\,\longrightarrow\, X$$ is also Nori-reduced. 
\end{lemma}

\begin{proof}
If $\pi(X)\,\longrightarrow\, H\times G$ is not faithfully flat then there
exists a triple $(M\, , Z\, ,\iota)$, where
\begin{itemize}
\item $M\,\hookrightarrow\, H\times G$ is a subgroup-scheme,

\item $Z\,\longrightarrow\, X$ is a $M$--torsor, and

\item $\iota\, :\, Z\,\hookrightarrow\, T\times_X Y$ is a reduction of structure
group-scheme, to $M$, of the $H\times G$--torsor $T\times_X Y$.
\end{itemize}
Let
$$T'\,\longrightarrow\, X~ \ \text{ and }~\ Y'\,\longrightarrow\, X$$ be the
$M^{\text{\'et}}$ and $M^{\text{loc}}$--torsors respectively,
obtained from the $M$--torsor $Z\,\longrightarrow\, X$ using the projections
of $M$ to $M^{\text{\'et}}$ and $M^{\text{loc}}$ respectively (the notation
is as in Lemma \ref{lemGROUPS}).
We have a closed immersion $$Z\,\hookrightarrow\, T'\times_X Y'$$ induced by
the closed immersion $M\,\hookrightarrow \,M^{\text{\'et}}\times M^{\text{loc}}$.
The latter is an isomorphism by Lemma \ref{lemGROUPS}, so the same is true for
$Z\,\hookrightarrow\, T'\times_X Y'$.

The projection $M\,\twoheadrightarrow\, H$ (respectively, $M\,\twoheadrightarrow
\, G$) clearly factors through $M^{\text{\'et}}$ (respectively, $M^{\text{loc}}$).
Note that the projections $$M\,\longrightarrow\, H \ \ 
\text{ and } \ \ M\,\longrightarrow \,G$$
are faithfully flat morphisms because the two torsors $Y$ and $T$ are
Nori-reduced. Consequently, the two homomorphisms $M^{\text{\'et}}\,\longrightarrow\,
H$ and $M^{\rm loc}\,\longrightarrow\, G$ are isomorphisms. Now using
Lemma \ref{lemGROUPS} it follows that the inclusion $M\,\hookrightarrow\, H\times G$
is an isomorphism. Consequently, the $H\times G$--torsor $T\times_X Y\,\longrightarrow\,
X$ is Nori-reduced.
\end{proof}

The following result was proved in \cite[Proposition 3.6]{EPS} under the assumption
that $X$ is proper.

\begin{corollary}\label{corTORS}
Let $X$ be a connected and reduced scheme over $k$. Let $G$ (respectively, $H$)
be a finite local (respectively, finite \'etale) $k$-group
scheme. Let $Y\,\longrightarrow\, X$ be a $G$--torsor and
$T\,\longrightarrow\, X$ an $H$--torsor. We assume that both the torsors are
Nori-reduced. Then the $G$--torsor
$$T\times_X Y \,\longrightarrow\,T$$ is also Nori-reduced. In particular, the morphism
$\pi^{\rm loc}(T)\,\longrightarrow\, \pi^{\rm loc}(X)$ is faithfully flat.
\end{corollary}

\begin{proof}
Let us assume that there is a finite local $k$-group
scheme $G_1\,\subset\,G$, and $G_1$--torsor $U\,\longrightarrow\, T$ and a
reduction
$$i\,:\,U\,\hookrightarrow\, T\times_X Y$$ of structure group to $G_1$.
Let $S$ be any $k$-scheme. For any $x\,\in\, X(S)$ we choose $u_x\,\in\, U(S)$ whose
image in $X(S)$ is $x$. We set $(t_x\, ,y_x)\,:=\,i(u_x)$, then $$T(S)\times_{X(S)}Y(S)
\,=\,\{(ht_x,gy_x)\, ,~\ \forall ~x\,\in\, X(S)\, ,~\ \forall ~ g\,\in\, G(S)\, ,
~\ \forall ~ h\,\in\, H(S) \}$$ so the image
of $U(S)$ by $i_S$ can be identified with the subset 
$$\{(ht_x\, ,gy_x)\, ,~\ \forall ~x\,\in\, X\, ,~\ \forall ~g\,\in\,
G_1\, ,~\ \forall ~h\,\in\, H \}\, ;$$ this gives $U$ the structure of an $H\times
G_1$--torsor over $X$, contained in the $H\times G$--torsor $T\times_X Y$. This
implies that $G_1\,=\,G$ by Lemma \ref{lemTORS}.
\end{proof}

\begin{corollary}\label{corTORS2}
Let $X$ be a connected reduced scheme over $k$ and $G$ a finite local $k$-group
scheme. Let $T\,\longrightarrow\, X$ be a finite \'etale cover, and let $Y
\,\longrightarrow\, X$ be a $G$--torsor. If
$Y\,\longrightarrow\, X$ is Nori-reduced and $T$ is connected, then the $G$--torsor
$$T\times_X Y\,\longrightarrow\, T$$ is also Nori-reduced. In particular the homomorphism
$\pi^{\rm loc}(T)\,\longrightarrow\, \pi^{\rm loc}(X)$ is faithfully flat.
\end{corollary}

\begin{proof} This follows from Corollary \ref{corTORS} and the fact that there exist a
finite \'etale $k$--group scheme $H'$ and an $H'$--torsor $T'\,\longrightarrow\, X$
that dominates $T\,\longrightarrow\, X$.
\end{proof}

\begin{remark}\label{remNORI}
In \cite{N2}, Nori proved that if
$i\,:\,U\,\longrightarrow\, Y$ is an open immersion between connected and
reduced schemes with $Y$ normal, then the morphism $\pi(U)\,\longrightarrow\,
\pi(Y)$ induced by $i$ is faithfully flat (see \S~II, Proposition 6 and
its corollaries). Consequently, the homomorphism
$\pi^{\rm loc}(U)\,\longrightarrow\, \pi^{\rm loc}(Y)$ induced by $i$ is
also faithfully flat. 
\end{remark}

\begin{notation}\label{notFINDEX}
Let $k$ be a field and $u\,:\,M\,\longrightarrow\, G$ a $k$-group scheme
homomorphism. We say that $u$ is of finite index if the following
property is satisfied: for any $k$-group scheme $Q$ and any faithfully
flat morphism of $k$-group schemes $G\,\longrightarrow\, Q$, if the group
scheme image of $M\to Q$ is finite then $Q$ is also finite.
\end{notation}

\begin{lemma}\label{lemCLinsep}
Let $f\,:\,X\,\longrightarrow\, Y$ be a finite purely inseparable morphism between normal
integral schemes. Then the homomorphism $$\pi(X)\,\longrightarrow\,
\pi(Y)$$ is of finite index, while $\pi^{\rm \acute{e}t}(X)\,\longrightarrow\,
\pi^{\rm \acute{e}t}(Y)$ is faithfully flat. So in particular $\pi^{\rm loc}(X)\,\longrightarrow\,
\pi^{\rm loc}(Y)$ is of finite index. 
\end{lemma}

\begin{proof}
We assume that $char(k)\,=\,p\,>\,0$. We observe that under the above assumptions the morphism
$f$ is surjective (see \cite[Ex. 5.3.9]{Li}). Let us first consider the case where $Y\,=\,X$ with $f\,:=\,F_Y$ being 
the absolute Frobenius morphism of $Y$. Let $T \,\longrightarrow\, Y$ be the universal 
$\pi(Y)$--torsor of $Y$ (it is a scheme, as all the transition morphisms
are affine), where $\pi(Y)$ is the fundamental group scheme of Nori. We set
$$T^{(p)}\,:= \, T\times_Y Y$$ via the Frobenius $F_Y$ of $Y$. As
usual, $F_{T/Y}\,:\, T\,\longrightarrow\, T^{(p)}$ is the 
relative Frobenius. The relative Frobenius commutes with base change, so if we 
pull back over $x\,:\,{\rm Spec}(k) \,\longrightarrow\, Y$ (a fixed closed point) what we obtain is the relative
Frobenius
$$ F_{\pi(Y)/{\rm Spec}(k)}\,:\, \pi(Y) \,\longrightarrow\, \pi(Y)^{(p)}
\,\simeq\, \pi(Y)\, ,$$
where the last isomorphism clearly follows from the fact that $k$ is algebraically closed, whence perfect; thus,
in particular, $$F_{T/Y}\,:\, T \,\longrightarrow\, T^{(p)}$$
is the natural morphism from the universal torsor to
the pro-finite torsor obtained after pulling back. The same holds for any
torsor: so let $P$ be a $Q$-torsor, where $\pi(Y)\,\longrightarrow\, Q$ is a
faithfully flat $k$-group scheme homomorphism; then we have the
relative Frobenius
$$ F_{Q/{\rm Spec}(k)}\,:\, Q \,\longrightarrow\, Q^{(p)}\,\simeq\, Q\, ,$$ which
factors as $Q\,\longrightarrow\, F\,\longrightarrow\, Q$ (where $Q\,
\longrightarrow\, F$ is faithfully flat and $F\,
\longrightarrow\, Q$ is a closed immersion). Since the kernel is finite, if
we assume $F$ to be finite then $Q$ is also finite thus
$F_{\pi(Y)/{\rm Spec}(k)}$ is of finite index.

Now $F_{\pi(Y)/{\rm Spec}(k)}$ is a finite endomorphism and this is sufficient to conclude that it is of finite index. As \'etale
torsors are not 
modified by the Frobenius, it follows that $ F_{\pi^{\rm \acute{e}t}(Y)/{\rm
Spec}(k)}$ is an isomorphism. What has been 
proved for $f\,=\,F_Y$ still holds, of course,
for $f\,=\,F_Y^m$, the Frobenius iterated $m$ times. So now we consider the general case where $f
\,:\,X\,\longrightarrow\,Y$ is the given purely 
inseparable morphism. Then there exist a positive integer $m$ and a morphism $h\,:\, Y\,\longrightarrow\, X$ such that 
$$f\circ h \,=\,F_Y^m \,:\, Y \,\longrightarrow\, Y$$
(the absolute Frobenius morphism iterated $m$ times). We consider the pullback 
$$T_X \,:=\, T\times_Y X$$ and the universal $\pi(X)$--torsor $P\,
\longrightarrow\, X$ on $X$. There are natural morphisms
$P\,\longrightarrow\, T_X$ and $$u\,:\,\pi(X)\,\longrightarrow\, \pi(Y)\, .$$
Pulling back further to $h\,:\, Y
\,\longrightarrow\,X$, the following factorization is obtained:
$$\xymatrix{\pi(Y) \ar@/_1pc/[rr]_{F^m_{\pi(Y)/{\rm Spec}(k)}} \ar[r] & \pi(X) \ar[r]^u & \pi(Y).}$$

The previous discussion yield the following:

\begin{itemize}
 \item if we assume that $\pi(Y)$ and $\pi(X)$ are both \'etale, this implies
that $$u
\,:\, \pi(X)\,\longrightarrow\, \pi(Y)$$ is faithfully flat;
 \item otherwise we can only conclude that $u$ is of finite index, which is all
that we can expect.
\end{itemize}
This is enough to conclude the proof.
\end{proof}

\begin{lemma}\label{lemCL2}
Let $f\,:\,X\,\longrightarrow\, Y$ be a dominant morphism between normal
integral schemes. Then the homomorphisms
$$\varphi^{\rm loc}:\pi^{\rm loc}(X)\,\longrightarrow\,
\pi^{\rm loc}(Y)\ \ \text{ and } \ \ \varphi^{\rm \acute{e}t}:\pi^{\rm \acute{e}t}(X)\,
\longrightarrow\,
\pi^{\rm \acute{e}t}(Y)$$ induced by $f$ are of finite index. 
\end{lemma}

\begin{proof}
This is inspired by \cite[Lemme 2]{CL} (see also \cite[Lemme 4.4.17]{De}
for the zero characteristic case). 

Let $t$ be a closed point of the generic fiber of $f$, and let $T$ denote its Zariski
closure in $X$. The morphism $f$ induces a generically finite morphism $f_{\vert T}\,
:\,T \longrightarrow\, Y$: indeed its generic fiber has relative dimension
zero and it this thus a finite number of points; therefore, there exists an
open dense subscheme $U\,\subseteq\, Y$ such that
$$f'\,:\,V\,\longrightarrow\, U\, ,$$
where $V\,:=\,T\times_Y U$, is a finite morphism. Hence there exist a scheme $W$ and two
finite morphisms
$$e\,:\,V\,\longrightarrow\, W\ \ \text{ and }\ \ i\,:\,W\,\longrightarrow\, V$$ such that $i\circ e\,=\,f'$, where $i$ is purely inseparable
and $e$ is generically \'etale. This implies
that there exists an open dense subscheme $W'\,\subseteq\, W$ such that
$$e'\,:\,V'\,\longrightarrow\, W'\, ,$$ where $V':=V\times_W W'$, is a finite \'etale cover. In what follows 
we study the morphism
$$\varphi^{\rm loc}\,:\,\pi^{\rm loc}(X)\,\longrightarrow\,
\pi^{\rm loc}(Y)\, ,$$ the (similar) details for 
$\varphi^{\rm \acute{e}t}$ are left to the reader. By Corollary
\ref{corTORS2} the morphism $\pi^{\rm loc}(V')\,\longrightarrow\,
\pi^{\rm loc}(W')$ induced by $e'$ is faithfully flat while the morphism 
$$\pi^{\rm loc}(W)\,\longrightarrow\, \pi^{\rm loc}(U)$$
induced by $i$ is of finite index by Lemma \ref{lemCLinsep} and clearly 
$\pi^{\rm loc}(W')\,\longrightarrow\, \pi^{\rm loc}(W)$ is faithfully flat (see Remark \ref{remNORI})
so the composition $$u:\pi^{\rm loc}(V')\,\longrightarrow\, \pi^{\rm loc}(U)$$ is of finite index. Now
we have the diagram of homomorphisms
of local fundamental group schemes 
$$\xymatrix{\pi^{\rm{loc}}(V')\ar[r]\ar[d]_{u} & \pi^{\rm{loc}}(X)\ar[d]^{\varphi^{\rm loc}}
\\ \pi^{\rm{loc}}(U)\ar[r]^{v} & \pi^{\rm{loc}}(Y). }$$
Now $u$ is of finite index and the homomorphism $v$ is
faithfully flat (see, again, Remark \ref{remNORI}). Hence $\varphi^{\rm loc}$ is finite index.
\end{proof}

In \cite{MS} Mehta and
Subramanian proved that $$\pi(X\times Y)\,=\, \pi(X)\times \pi(Y)$$ for two connected, proper and reduced schemes 
$X$ and $Y$. If one of the two schemes ($X$ or $Y$) is not proper anymore then the previous formula may not hold. 
However a weaker result will be sufficient for our purposes; the following proposition can be found in 
\cite{N2}, Chapter II, Proposition 9, here we suggest a different approach :

\begin{proposition}\label{propNEWPROD}
Let $Y$ be an integral scheme over $k$, then the homomorphism $$\varphi\,:\,
\pi(\mathbb{P}^1\times Y)\,\longrightarrow\, \pi(Y)$$ induced by 
the projection $p_2\,:\,\mathbb{P}^1\times Y\,\longrightarrow\, Y$ is an
isomorphism.
\end{proposition}

\begin{proof}
It is clear that 
$\varphi$ is faithfully flat as $p_2$ has a section, so we only need to prove
that given a finite $k$--group scheme 
$G$ and a $G$--torsor $T\,\longrightarrow\, \mathbb{P}^1\times Y$, there exists a
$G$--torsor $$T'\,\longrightarrow\,
U$$ whose pullback to $\mathbb{P}^1\times Y$ is the given one. First we briefly recall that the fundamental group scheme of $Y$ at a $k$-point $y$ is the automorphism group scheme of the fiber functor $y^{\ast}$ on the category of essentially finite vector
bundles, as described in \cite{N1}. Let $Rep_k(G)$ 
denote the category of $k$--linear finite dimensional representations of $G$. Then 
associated to our $G$--torsor $T\,\longrightarrow\, \mathbb{P}^1\times Y$ there is a
fiber functor 
$$F_T\,:\,Rep_k(G)\,\longrightarrow\, \mathcal{Q}coh(\mathbb{P}^1\times Y)$$ by a
fundamental result in 
Tannakian theory (recalled for instance in \cite[Proposition (2.9)]{N1}). From this 
we will construct a functor $$F\,:\,Rep_k(G)\,\longrightarrow\,
\mathcal{Q}coh(Y)\, .$$ For 
any $G$--module $V$, set $$F(V)\,=\,(p_2)_*(F_T(V))\, .$$ We first observe that $F(V)$ is a 
vector bundle: when restricted to $\mathbb{P}^1$, clearly $F_T(V)$ is an essentially finite 
vector bundle over the projective line, thus trivial, whence $H^1(\mathbb{P}^1, 
F_T(V))\,=\,0$, and the evaluation homomorphism $$(p_2)^*(p_2)_* (F_T(V))
\,\longrightarrow\, (F_T(V))$$ is an isomorphism.
Moreover $F$ is compatible with the operations of taking tensor 
products, direct sums and duals. Hence $F$ is a fiber functor and we can 
associate to it a $G$--torsor $T'\,\longrightarrow\, U$ which is the desired one
since its pullback to $\mathbb{P}^1\times Y$ is isomorphic to
$T\,\longrightarrow\, \mathbb{P}^1\times Y$.
\end{proof}

We now recall that for any $0\, \leq\, \,\nu\, \leq\, \, \dim(X)$, 
there is a point in $X(\Omega)$, where $\Omega$ in the
algebraic closure of the function field of $X$, whose Zariski closure in $X$
is of dimension $\nu$.

\begin{theorem}\label{teoMAIN}Let $k$ be an algebraically closed field and $X$
a normal, rationally chain connected $k$--scheme. Then $\pi^{\rm loc}(X)$ is finite.
\end{theorem}

\begin{proof}
Since $X$ is rationally chain connected, there exists a chain of rational curves
connecting a rational point $x_0\,\in\, X(k)$ to a generic point $x_m\,\in\, X(\Omega)$,
where $\Omega$ is the algebraic closure of the function field of $X$. According to Lemma
\ref{lemCL1} there exists a sequence of integral subvarieties $V_0\, , \cdots\, , V_m$
of $X$ where $V_0\,=\,x_0$ and $V_m\,=\,X$ and for every integer $i\,\in\,
\{0, \cdots , m-1\}$ a family of rational curves
$$
F^i\,:\,\mathbb{P}^1_k\times T_i \,\longrightarrow\, X
$$
with $T_i$ normal and projective, such that the morphisms
$$F^i_0\,:\,T_i\,\longrightarrow
\, X \ \ \text{ and }\ \ F^i_{\infty}\,:\, T_i\,\longrightarrow\, X\, ,$$
defined by $F^i_0(t)\,:=\,
F^i(0,t)$ and $F^i_{\infty}(t)\,:=\,F^i(\infty,t)$, are dominant on $V_i$ and $V_{i+1}$
respectively. If $V_i$ is not normal then we can consider an open normal
subscheme $V_i'\,\subset\, V_i$ and the pullback 
$$\xymatrix{T_i'\ar[r]\ar[d] & V_i'\ar[d] \\ T_i\ar[r] & V_i. }$$
In a similar way, if $V_{i+1}$ is not normal then we can consider an open normal
subscheme $V_{i+1}'\,\subset\, V_{i+1}$ and its pullback, as before, that we will call
$T_i^{''}$. This will not affect $V_0$ and $V_m$ of course. So this induces the
following commutative diagram on local group schemes:
$$\xymatrix{ & & \pi^{\text{loc}}(T_i')\ar[ld]_{\alpha}\ar[rd]^{u}
 & & \\ & \pi^{\text{loc}}(T_i)\ar[ld]_{\beta} & &
\pi^{\text{loc}}(V_i')\ar[dr]^{v} & \\ \pi^{\text{loc}}(\mathbb{P}^1_k\times
T_i) \ar[rrrr]^{\pi(F^i)} & & & & \pi^{\text{loc}}(X) \\ &
\pi^{\text{loc}}(T_i)\ar[lu]^{\gamma} & & \pi^{\text{loc}}(V_{i+1}')\ar[ru]_{w}
& \\ & & \pi^{\text{loc}}(T_i^{''})\ar[lu]^{\delta}\ar[ru]_{z} & & }$$
We avoid to put the index $i$ on the morphisms not to make notation too
heavy. We know that $\pi^{\text{loc}}(V_0)\,=\,0$, both $u$ and $z$ are of finite index by Lemma \ref{lemCL2}, 
both $\alpha$ and $\delta$ are faithfully flat by 
Remark \ref{remNORI} and both $\beta$ and $\gamma$ are isomorphisms by Proposition 
\ref{propNEWPROD}. So at each step we prove that the image of
$\pi^{\text{loc}}(V_{i+1}')$ in $\pi^{\text{loc}}(X)$ is finite. The last step
will finally prove that $\pi^{\text{loc}}(X)$ is finite.
\end{proof}

\begin{lemma}\label{lemLAST}Let $X$ be a rationally chain connected variety and let $f : Y \longrightarrow X$ be an \'etale Galois cover. Then $Y$ is rationally chain connected.
\end{lemma}

\begin{proof}
Fix a point $y$ of $Y$. Let $U_y$ be the subset of $Y$ that is rationally chain connected to the point $y$. Let $\Omega$ be any algebraically closed field containing $k$, then any morphism  $g : \mathbb{P}^1_{\Omega}\longrightarrow X_{\Omega}$ lifts to a morphism
$g' : \mathbb{P}^1_{\Omega}\longrightarrow Y_{\Omega}$ using the homotopy lifting property  because $\mathbb{P}^1_{\Omega}$ is simply connected. Since $X$ is also rationally chain connected, these two together imply that
$U_y$ is both open and closed. Hence $U_y = Y$, and $Y$ is rationally chain connected.
\end{proof}

\begin{remark}\label{remFANO}
Let notations be as in Theorem \ref{teoMAIN}, then from \cite[Th\'eor\`eme]{CL} and Lemma \ref{lemLAST} we also obtain that $\pi(X)$ is finite. If $char(k)\,=\,p\,>\,0$, and $X$ is moreover smooth and proper, then $\vert\pi(X)^{\rm \acute{e}t}\vert$ is coprime to $p$, as proved in \cite{CL2}. In particular all this holds when $X$ is a Fano variety since in this case it is rationally chain connected (cf. \cite{Ca} and \cite{KMM}). Furthermore when $X$ is a general hypersurface of a projective space then then it is separably
rationally connected, and by \cite{Bi} this implies that $\pi(X)\,=\,0$.
\end{remark}

\section*{Acknowledgments}

We thank Antoine Chambert-Loir for a useful communication. We thank the two
referees for comments that helped us in improving the paper. The first-named author would like to thank T.I.F.R. for 
its hospitality and Cinzia Casagrande for useful discussion. The second-named author 
thanks Universit\'e Lille 1 and Niels Borne for hospitality. He also acknowledges the 
support of a J. C. Bose Fellowship.

\end{document}